\documentclass[onefignum,onetabnum]{siamart171218}

\usepackage{lipsum}
\usepackage{amsfonts}
\usepackage{graphicx}
\usepackage{epstopdf}
\usepackage{algorithmic}
\usepackage{amsmath}
\usepackage{amssymb}
\usepackage{amsfonts}
\usepackage{dsfont}
\ifpdf
  \DeclareGraphicsExtensions{.eps,.pdf,.png,.jpg}
\else
  \DeclareGraphicsExtensions{.eps}
\fi

\usepackage{subcaption}

\newcommand{\ga}{\alpha}
\newcommand{\gb}{\beta}

\newcommand{\gep}{\epsilon}

\newcommand{\gl}{\lambda}
\newcommand{\go}{\omega}
\newcommand{\gs}{\sigma}

\newcommand{\gO}{\Omega}

\newcommand{\cB}{\mathcal{B}}

\newcommand{\cE}{\mathcal{E}}
\newcommand{\cF}{\mathcal{F}}

\newcommand{\N}{\mathbb{N}}

\newcommand{\R}{\mathbb{R}}

\newcommand{\OAP}{$(\Omega,\mathcal{A},P)$}

\newcommand{\Ind}{\mathds{1}}

\newcommand{\lanO}{\mathrm{O}}

\newcommand{\BV}{\operatorname{BV}}
\newcommand{\TV}{\operatorname{TV}}

\newcommand{\BIGOP}[1]{\mathop{\mathchoice%
{\raise-0.22em\hbox{\huge $#1$}}%
{\raise-0.05em\hbox{\Large $#1$}}{\hbox{\large $#1$}}{#1}}}

\newcommand{\BIGboxplus}{\mathop{\mathchoice%
{\raise-0.35em\hbox{\huge $\boxplus$}}%
{\raise-0.15em\hbox{\Large $\boxplus$}}{\hbox{\large $\boxplus$}}{\boxplus}}}

\newsiamremark{remark}{Remark}
\newsiamremark{hypothesis}{Hypothesis}
\crefname{hypothesis}{Hypothesis}{Hypotheses}
\newsiamthm{claim}{Claim}

\headers{PDMP driven by scalar conservation laws}{S.\ Knapp}

\title{Piecewise deterministic Markov processes driven by scalar conservation laws\thanks{Submitted on January 14, 2019.
\funding{Financially supported by the BMBF project ENets (05M18VMA) and DAAD-PPP USA (Project-ID 57444394).}}}

\author{Stephan Knapp\thanks{Department of Mathematics, University of Mannheim, Mannheim, Germany 
  (\email{stknapp@mail.uni-mannheim.de}).}
  }

\usepackage{amsopn}

\ifpdf
\hypersetup{
  pdftitle={Piecewise deterministic dynamics driven by scalar conservation laws},
  pdfauthor={S.\ Knapp}
}
\fi

\begin{document}

\maketitle

\begin{abstract}
We investigate piecewise deterministic Markov processes (PDMP), where the deterministic dynamics follows a scalar conservation law and random jumps in the system are characterized by changes in the flux function. We show under which assumptions we can guarantee the existence of a PDMP and conclude bounded variation estimates for sample paths. Finally, we apply this dynamics to a production and traffic model and use this framework to incorporate the well-known scattering of flux functions observed in data sets.
\end{abstract}


\section{Introduction}
The simplicity of scalar conservation laws allows to understand general behaviors of underlying models but, on the other hand, they are based on qualified assumptions as for example steady state or expected values. One possibility to widen this class of models are systems of conservation laws, where fluctuations and higher order moments can be governed. Another possibility to extend scalar conservation laws are stochastic effects. More precisely, starting from deterministic scalar conservation laws and a corresponding initial value problem (IVP)
\begin{align}
u_t(x,t)+f(u(x,t))_x = 0, \quad u(x,0) = u_0(x),\label{eq:IVPDet}
\end{align}
a natural extension is the incorporation of uncertainties. There already exist extensions based on a reformulation as stochastic differential equation like in \cite{Holden1997} and partial stochastic differential equation as in \cite{Feng2008,Lions2014} in the literature. Also uncertain initial data as for example in \cite{Fjordholm2017} and random chosen flux functions  \cite{Mishra2016} have been considered. In the latter work, the flux function is random and does not change randomly in time.

In contrast to  \cite{Mishra2016}, our goal is a stochastic process, which ``chooses'' a new flux function at random times, where these times and the random choice of the next flux function may dependent on the actual solution of the whole system. This can be easily motivated by, e.g.\ production models with machine failures \cite{GoettlichKnapp2017,GoettlichKnapp2018}, and also opinion formation, change of state (gas to liquid or vice versa) are reasonable applications.
This idea directly transfers us into the theory of piecewise deterministic Markov processes, see \cite{Jacobsen2006}.
In detail, given a parametrized family of Lipschitz continuous flux functions $f^\ga \in C^{0,1}(\R)$ for $\ga \in I \subset \R$, we are interested in a ``solution'' to
\begin{align}
u_t(x,t)+f^{\alpha(t)}(u(x,t))_x = 0, \quad u(x,0) = u_0(x),\label{eq:IVPStoch}
\end{align}
where $\ga(t) \in I$ denotes the current and random chosen flux function at time $t \in [0,T]$ and $x \in \R$.

We define how \eqref{eq:IVPStoch} has to be understood and how $\ga(t)$ is specified in the subsequent section \ref{sec:ModEq}. This section is followed by applications and numerical results in the case of a production and traffic model in section \ref{sec:Applications}.

\section{Modeling Equations}\label{sec:ModEq}
Let $u \colon \R \to \R$ be a function, then we denote by $\TV(u)$ its total variation and define $\BV(\R) = \{u \colon \R \to \R\colon \TV(u)<\infty\}$ as the set of all functions from $\R$ to $\R$ with total bounded variation, see, e.g.\ \cite{Holden2015}.
With this notation, it is well known as a result of Krusckov, see \cite{Holden2015}, that the IVP \eqref{eq:IVPDet} has a unique weak entropy solution if $u_0 \in \BV(\R)\cap L^1(\R)$ and if $f \in C^{0,1}(\R)$, i.e., is Lipschitz continuous. Furthermore, the solution $u$ satisfies
\begin{align}
\|u(\cdot,t)\|_\infty &\leq \|u_0\|_\infty,\\
\TV(u(\cdot,t))&\leq \TV(u_0), \\
\|u(\cdot,t)-u(\cdot,s)\|_{L^1} &\leq \|f\|_{C^{0,1}} \TV(u_0)|t-s|
\end{align}
and is $L^1$ stable with respect to initial data
\begin{align}
\|u(\cdot,t)-v(\cdot,t)\|_{L^1} \leq \|u_0-v_0\|_{L^1}.
\end{align}

\subsection*{Deterministic dynamics between jump times}
In this section, we define the dynamics to \eqref{eq:IVPStoch} based on theory of PDMPs. Unfortunately, we cannot apply the theory of PDMPs directly on solutions to \eqref{eq:IVPDet} with corresponding flux functions $f^\ga$ since $\BV(\R)$ is not separable and hence no Borel space. Following \cite{Bressan2000}, we use the extended solution operator to \eqref{eq:IVPDet} on $L^1(\R)$ and denote it by $S_t^\ga \colon L^1(\R) \to L^1(\R)$, where $\ga$ indicates that flux function $f^\ga$ is used.
We directly deduce the following properties of the family $(S^\ga_t,t\in [0,T])$ for every $\ga \in I$:
\begin{align}
S^\ga_{s+t} &= S^\ga_sS^\ga_t = S^\ga_t S^\ga_s \text{ for } s,t \in [0,T] \text{ with } s+t \in [0,T],\label{eq:S_Semigroup}\\
S^\ga_0 &= Id,\label{eq:S_Identity}\\
t \mapsto S^\ga_t &\in C([0,T];L^1(\R)),\label{eq:S_Continuous}\\
\|S_t^\ga u-S_t^\ga v\|_{L^1} &\leq \|u-v\|_{L^1},\label{eq:S_L1Contraction}\\
t \mapsto S^\ga_tu_0 \text{ is the } &\text{unique entropy solution to \eqref{eq:IVPDet} if } u_0 \in L^1(\R)\cap \BV(\R).
\end{align}

Up to now, we have no specification of $\ga(t)$. We define the state space $E = L^1(\R)\times I$ equipped with the Borel $\sigma$-algebra $\cE$ generated by the open sets induced by $\|(u,\ga)\| = \|u\|_{L^1}+|\ga|$ for $(u,\ga) \in E$. Then $(E,\cE)$ is a Borel space.

Our aim is to switch the flux function only at random times, which results in deterministic dynamics between the jumps in the form of
\begin{align*}
\phi_t \colon E \to E, \quad
\begin{pmatrix}
u\\
\ga
\end{pmatrix}
 \mapsto \begin{pmatrix}
S^\ga_tu,\\
\ga
\end{pmatrix}
.
\end{align*} 
Properties \eqref{eq:S_Semigroup}-\eqref{eq:S_L1Contraction} of $S$ directly translate to $\phi$. If we can show that $\phi \colon [0,T]\times E \to E$ is measurable, the dynamics $\phi$ is a candidate for deterministic dynamics in between jump times of a PDMP, see \cite{Jacobsen2006}. The following lemma \ref{lem:PhiContinuous} tells us a sufficient condition to prove measurability of $\phi$.

\begin{lemma}\label{lem:PhiContinuous}
Let the mapping $\ga \mapsto f^\ga$ from $I \to C^{0,1}(\R)$ be continuous with $I \subset \R$ an interval, then $(t,u,\ga) \mapsto (S^\ga_t u,\ga)$ is continuous from $[0,T]\times L^1(\R) \times I \to L^1(\R)\times I$ and consequently measurable.
\end{lemma}
\begin{proof}
Let $(s,u,\ga),(t,v,\gb) \in [0,T]\times L^1(\R) \times I$, then we use the norm
\begin{align*}
\|(s,u,\ga)-(t,v,\gb)\| = |s-t|+\|u-v\|_{L^1}+|\ga-\gb|.
\end{align*}
According to this norm, we use
\begin{align*}
\|(S_s^\ga u,\ga)-(S_t^\gb v,\gb)\| = \|S_s^\ga u-S_t^\beta v\|_{L^1}+|\ga-\gb|.
\end{align*}
To show continuity, we estimate $\|S_s^\ga u-S_t^\beta v\|_{L^1}$ as follows:
\begin{align*}
\|S_s^\ga u-S_t^\beta v\|_{L^1} & \leq \|S_s^\ga u-S_t^\ga u\|_{L^1}+\|S_t^\ga u-S_t^\ga v\|_{L^1}+\|S_t^\ga v-S_t^\gb v\|_{L^1}
\end{align*}
and conclude that we can make $\|S_s^\ga u-S_t^\ga u\|_{L^1}$ and $\|S_t^\ga u-S_t^\ga v\|_{L^1}$ sufficiently small by shrinking $\|(S_s^\ga u,\ga)-(S_t^\gb v,\gb)\|$ due to properties \eqref{eq:S_Continuous}-\eqref{eq:S_L1Contraction}.
Let $(v_n, n \in \N)$ be a sequence in $\BV(\R)\cap L^1(\R)$ satisfying $\lim_{n \to \infty} \|v_n-v\|_{L^1} = 0$. We estimate $\|S_t^\ga v-S_t^\gb v\|_{L^1}$ as follows:
\begin{align*}
\|S_t^\ga v-S_t^\gb v\|_{L^1} & \leq \|S_t^\ga v-S_t^\ga v_n\|_{L^1}+\|S_t^\gb v_n-S_t^\gb v\|_{L^1}+\|S_t^\ga v_n-S_t^\gb v_n\|_{L^1}\\
&\leq 2\|v-v_n\|_{L^1}+\|S_t^\ga v_n-S_t^\gb v_n\|_{L^1}\\
&\leq 2\|v-v_n\|_{L^1} + t \|f^\ga-f^\gb\|_{C^{0,1}} \TV(v_n),
\end{align*}
where we used the result from \cite[p.\ 53]{HoldenRisebro} in the last estimate.
Altogether, we find that
\begin{align*}
\|S_s^\ga u-S_t^\beta v\|_{L^1} & \leq \|S_s^\ga u-S_t^\ga u\|_{L^1}+\| u- v\|_{L^1}\\
&\quad +2\|v-v_n\|_{L^1} + T \|f^\ga-f^\gb\|_{C^{0,1}} \TV(v_n).
\end{align*}
Now, let $(t,v,\gb) \in [0,T]\times L^1(\R) \times I$, $\gep>0$ and choose $n \in \N$ such that $\|v-v_n\|_{L^1} < \frac{\gep}{6}$ as well as $\delta>0$ such that
\begin{align*}
\|S_s^\ga u-S_t^\ga u\|_{L^1} < \frac{\gep}{6}, \quad \| u- v\|_{L^1} <\frac{\gep}{6},\quad \|f^\ga-f^\gb\|_{C^{0,1}} < \frac{\gep}{\TV(v_n) 6 T},\quad |\ga-\gb| < \frac{\gep}{6}
\end{align*}
implying $$\|(S_s^\ga u,\ga)-(S_t^\gb v,\gb)\| < \gep$$
for all $(s,u,\ga) \in [0,T]\times L^1(\R) \times I$ satisfying $\|(s,u,\ga)-(t,v,\gb)\| < \delta$.
\end{proof}

One simple example for a family of flux functions, which satisfies the continuity with respect to the parameter $\ga \in I$ is given by 
$
f^\ga = \ga f
$
for $f \in C^{0,1}(\R)$. Then $\|f^\ga-f^\gb\|_{C^{0,1}} = \|f\|_{C^{0,1}}|\ga-\gb|$.

\subsection*{Jump and jump time distributions}
Following \cite{Jacobsen2006}, we specify the transition intensities $q_t(y,B) \geq 0$, i.e., the rate to jump from $y \in E$ in a state in $B \in \cE$ at time $t \in [0,T]$. This can be decomposed into $q_t(y,B) = \eta_t(y,B)\psi_t(y)$, where $\psi_t(y)$ is the total intensity that a jump occurs a time $t$ and $\eta_t(y,B)$ is the probability of a jump from $y$ into a state in $B$ provided a jump occurs at time $t$.

In order to use these intensities, we assume $(y,t) \mapsto \psi_t(y)$ to be measurable and for all $(y,t)$ we need $\int_t^{t+h}\psi_s(y)ds<\infty$ for $h = h(y,t)$ sufficiently small.
For all $t$ we additionally assume that $\eta_t$ is a Marovian kernel, see, e.g.\ \cite{BauerWTengl}, for a definition.
A further and natural assumption is that $\eta_t(y,\{y\}) = 0$ holds for all $(y,t) \in E\times[0,T]$.

At this point almost everything can happen at jump times but we fix the specific idea that the flux function only changes at the jump times. In detail, there is no jump in the solution of the conservation law component to inherit mass conservation again.
To do so, we restrict on rates 
\begin{align*}
\gl \colon I\times \cB(I)\times [0,T] \times L^1(\R) \to \R_{> 0},
\end{align*}
satisfying
\begin{enumerate}
\item $\sup\{\gl(\ga,I,t,u)\colon \ga \in I,\; t \in [0,T],\; u \in L^1(\R)\}\leq  \gl^{\text{max}}<\infty$,
\item for every $t \in [0,T]$, $(\ga,u) \in E$ the mapping $B \mapsto \gl(\ga,B,t,u)$ is a measure,
\item for every $t \in [0,T]$, $B \in \cB(I)$ the mapping $(\ga,u) \mapsto \gl(\ga,B,t,u)$ is measurable,
\item for every $t \in [0,T]$, $(\ga,u) \in E$ we have $\gl(\ga,\{\ga\},t,u) = 0$.
\end{enumerate}
Then we define for every $y = (\ga,u) \in E$ and $B \in \cE$ the total intensity and jump distribution by
\begin{align*}
\psi_t(y) &= \gl(\ga,I,t,u),\\
\eta_t(y,B) &= \frac{1}{\gl(\ga,I,t,u)} \int_I \Ind_B((\gb,u)) \gl(\ga,d\gb,t,u).
\end{align*}

\subsection*{Existence}
Due to the uniform bound on $\psi_t$, we can use a so-called thinning algorithm to build the jump times $T_n$ and after jump locations $Y_n$ for $n \in \N_0$ iteratively, see \cite{GoettlichKnapp2018,Lemaire2017}. Since the number of jumps is finite $P$-almost surely, again due to the uniform bound on the rates, we obtain a stable random counting measure and theorem 7.3.1 from \cite{Jacobsen2006} can be applied. We obtain the following result
\begin{theorem}\label{thm:Existence}
For every initial data $x_0 = (\ga_0,u_0) \in E$ there exists a stochastic process $X=(X(t), t\in [0,T])$ on some probability space \OAP, which satisfies
\begin{enumerate}
\item $X(0) = x_0$,
\item $X$ is a Markov process with respect to its natural filtration $\cF^X = (\cF^X_t,t \in [0,T])$ given by $\cF^X_t = \sigma(X(s),0\leq s\leq t)$,
\item $X$ is piecewise deterministic and piecewise continuous, i.e., there exist jump times $T_n \in [0,T]$ and post jump locations $Y_n \in E$ for $n \in \N_0$ such that $$X(t) = \phi_{t-T_n}(Y_n) \quad \Leftrightarrow \quad t \in [T_n,T_{n+1}),$$
where for convenience $T_0 = 0$ and $Y_0 = x_0$.
\end{enumerate}
\end{theorem}

\subsection*{Total Variation bounds and BV solutions}
The extension of the solution to $L^1$ allowed us to use classical results from the theory of piecewise deterministic Markov processes to obtain the existence of a stochastic process, which satisfies our requirements. 
We expect that if the initial condition $u_0 \in L^1(\R)\cap \BV(\R)$, then we deduce $u(t) \in  L^1(\R)\cap \BV(\R)$
again as the following lemma shows.
\begin{lemma}\label{lem:BV}
Let $X = (X(t),t \in [0,T])$ be the stochastic process from theorem \ref{thm:Existence} with $X(t) = (\ga(t),u(t)) \in E$. If $u(0) = u_0 \in L^1(\R)\cap \BV(\R)$, then $u(t) \in L^1(\R)\cap \BV(\R)$ and $\TV(u(t)) \leq \TV(u_0)$. 
\end{lemma}
\begin{proof}
Let $\go \in \gO$, $T_n(\go)$ the jump times and $Y_n(\go)$ the post jump locations of $X(\go)$ for $n \in \N_0$. 
For $t \in [0,T_1(\go))$ we have $\TV(u(t,\go)) = \TV(S^{\ga_0}_t u_0) \leq \TV(u_0)$ by classical results on scalar conservation laws, see, e.g.\ \cite{Holden2015}. 
At time $t = T_1$ the flux function changes and for $t \in [T_1(\go),T_2(\go))$ it follows 
$$\TV(u(t)) = \TV(S^{\ga(T_1(\go),\go)}_{t-T_1(\go)} u(t,\go)) \leq \TV(u(T_1(\go),\go)) \leq \TV(u_0)$$
by continuity of $t \mapsto u(t,\go)$. 
Iteratively, we deduce
\begin{align*}
\TV(u(t,\go)) \leq \TV(u_0).
\end{align*}
\end{proof}

\begin{remark}
Lemma \ref{lem:BV} is only valid because we have no jumps in the $u$ component at jump times by construction. Using the same arguments, the mass in the $u$ component is preserved.
\end{remark}

\section{Applications and numerical results}\label{sec:Applications}
Since we motivated PDMPs driven by scalar conservation law dynamics by the scattering of real data, we discuss simulation results of two examples in this section. The first example is a production and the second example is a traffic flow model.

\subsection*{Production model}
Macroscopic production models have been widely studied in the literature, see \cite{ApiceGoettlichHertyBenedetto} for an overview. Since in production capacity drops occur due to machine failures or human influences, deterministic models have been extended to stochastic production models, see \cite{DegondRinghofer2007, GoettlichKnapp2017,GoettlichMartinSickenberger,  GoettlichKnapp2018}. Therein, a random flux function in the form of $$f(\rho) = \min\{v\rho,\mu\}$$ has been chosen with a deterministic production velocity $v>0$, a stochastic capacity $\mu$ for a production density $\rho$. The latter corresponds to the variable $u$ in our context. In \cite{DegondRinghofer2007,GoettlichMartinSickenberger} the capacity $\mu$ is a Continuous Time Markov Chain, in \cite{GoettlichKnapp2017} a semi-Markov process and in \cite{GoettlichKnapp2018} a PDMP construction has been developed.

In contrast to the mentioned works, we consider a single production step instead of a network and use our more general setting that allows for further flux functions motivated by data sets, see e.g.\ \cite{Forestier2015}.
One possible choice is
\begin{align*}
f^\ga(\rho) = \mu(\ga) (1-e^{-\frac{v(\ga)}{\mu(\ga)}\rho})
\end{align*}
for a  continuous bounded capacity $\mu>0$ and velocity $v \geq 0$. Some calculation shows $\|f^\ga-f^\gb\|_{C^{0,1}} = \lanO(|v(\ga)-v(\gb)|+|\frac{v(\ga)}{\mu(\ga)}-\frac{v(\gb)}{\mu(\gb)}|)$ and the flux function fulfills the requirements to obtain the existence of a suitable stochastic process $X$, see theorem \ref{thm:Existence}.

In Figure \ref{fig:FluxFunctionsProduction1} flux functions for $\mu(\ga) = 1+\tanh(\frac{\ga}{2})$ and $v(\ga) = 1+\tanh(\ga)$ and different $\ga$ are drawn. So, we can capture different production velocities and capacities by varying $\ga$. 

\begin{figure}[htb!]
\begin{subfigure}[c]{0.5\textwidth}
\includegraphics[width=0.95\textwidth]{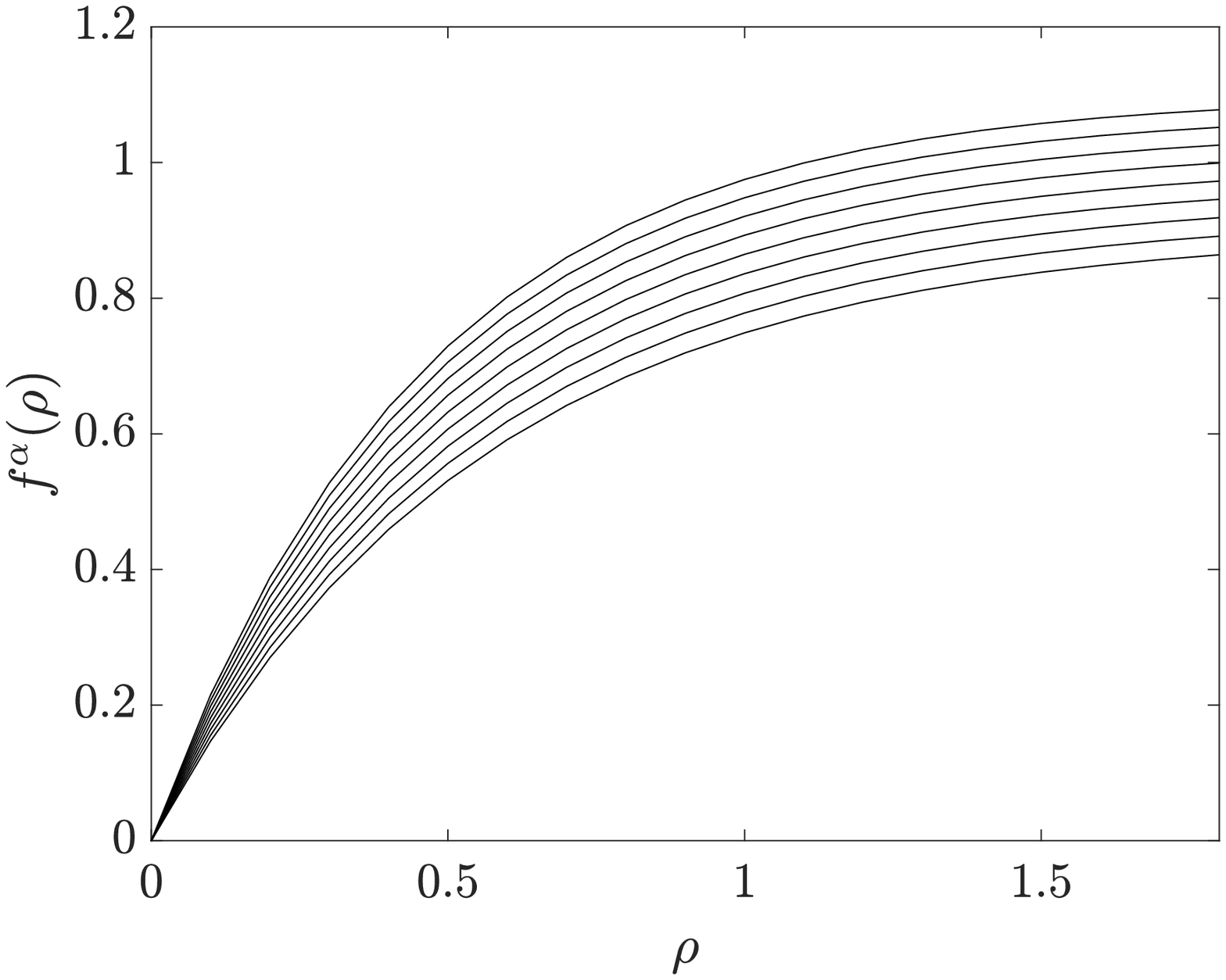}
\subcaption{Flux functions for different $\ga$}
\label{fig:FluxFunctionsProduction1}
\end{subfigure}
\begin{subfigure}[c]{0.5\textwidth}
\includegraphics[width=0.95\textwidth]{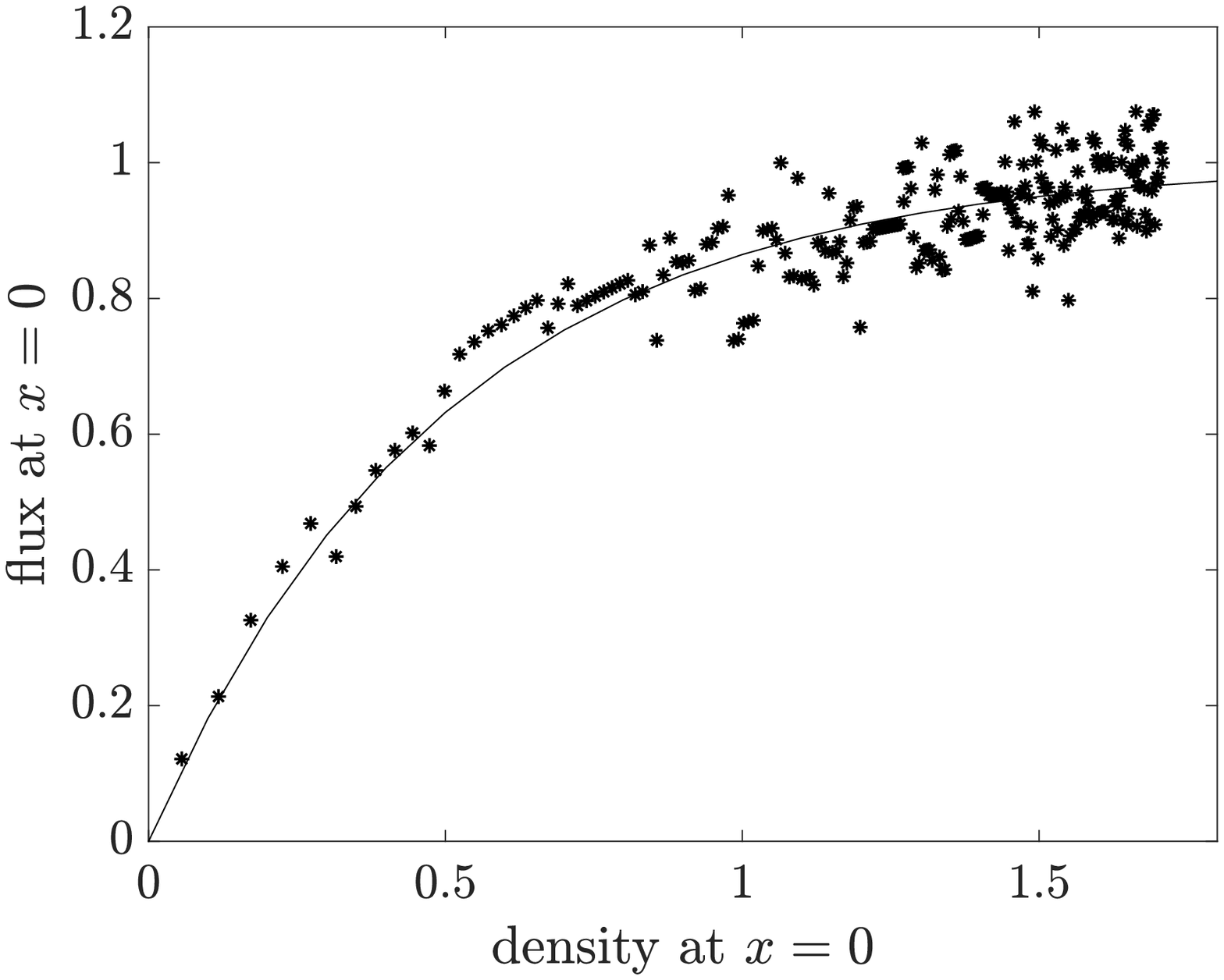}
\subcaption{Sample densities and fluxes at $x = 0$}
\label{fig:SampleDensFlux1}
\end{subfigure}
\end{figure}
It remains to introduce jump rates $\gl$ in the production setting. We want the total jump intensity to be dependent on the Work In Progress (WIP) on some interval $[a,b]\subset \R$, which is defined as
$
\text{WIP}(\rho(t)) = \int_a^b \rho(x,t) dx.
$
In detail, we assume as WIP increases, the probability of a change of the flux function increases and vice versa. The distribution of the post jump location is assumed to be symmetrical around $\bar{\ga}\in \R$ with variance $\gs^2>0$ and we exemplary use
\begin{align*}
\gl(\ga,B,t,\rho) = \bar{\gl}(\rho) \int_B \frac{1}{\sqrt{2 \pi \sigma^2(\rho)}} e^{-\frac{(z-\bar{\ga})^2}{2 \sigma^2(\rho)}}dz
\end{align*} 
for every $\ga \in \R$, $B \in \cB(\R)$, $t \in [0,T]$ and $\rho \in L^1(\R)$. One reasonable choice for $ \bar{\gl}(\rho) $ is $ \bar{\gl}(\rho)  = \gl_0 (1-e^{-\gl_1 \text{WIP}(\rho)})$ for some $\gl_0,\gl_1>0$. 
For the subsequent simulation results, we assume $a = 0$, $b = 1$, $\gl_0 = 5$, $\gl_1 = 1$, $\sigma^2 = 10^{-2}$, $\bar{\ga} = 0$. The time horizon is $T = 50$ and the numerical spatial domain is taken as large that boundary conditions have no influence at $x=0$ on the solution. The deterministic dynamics is approximated by a Godunov scheme and in figure \ref{fig:SampleDensFlux1} we see the result of one sample of the density flux relation at position $x = 0$ generated by the model with initial data 
$\rho(x,0) = \frac{3}{2}(\sin(x)+1)e^{-\frac{|x|}{100}}$. The black markers consider to the density and flux at times $t = 0,0.2,\dots,50$ and the black solid line in figure \ref{fig:SampleDensFlux1} represents the flux function for $\ga = 0$. We observe in this stochastic macroscopic production model the typical scattering effect like it is the case for microscopic production models driven by discrete event simulations in \cite{Forestier2015}.

\subsection*{Traffic flow model}
The scattering effect in the density flux diagram obtained by real data, see, e.g.\ \cite{Piccoli2012,Seibold2013}, is a fundamental pattern and important for the development of second order, stochastic and phase transition traffic flow models. In the so-called free phase we observe small fluctuations and an almost linearly increasing flux with respect to the density. At a critical density, the flux decreases in the so-called congested phase. The critical density and congested phase are characterized by higher variances, i.e.\ sacttering effects in data. There exist already stochastic approaches like in \cite{Li2012,Ni2018} and a comprehensive overview is given in \cite{Wang2011}. We will show that the framework, which we introduced in section \ref{sec:ModEq} is able to capture the scattering effects as well. 

As family of flux functions, we use, motivated by the shape of the probability density function of the Gamma distribution, 
\begin{align*}
f^\ga(\rho) = \frac{\theta-1}{\ga^\theta} \frac{1}{\Gamma(\frac{\theta-1}{\ga})}\rho^{\theta-1}e^{-\frac{\theta-1}{\ga}\rho}
\end{align*}
for some parameter $\theta\geq 1$, $\ga >0$, $\rho \geq 0$ and $\Gamma$ the Gamma function. If $\theta \geq 2$, we also have $f^\ga \in C^{0,1}(\R_{\geq 0})$ and the maximum is attained at $\rho^\ast = \alpha$.
In figure \ref{fig:FluxFunctionsLWR1}, we see the shape of the flux function by varying $\ga \in [0.3,0.5]$ and $\theta = 2.1$. We set
\begin{align*}
\gl(\ga,B,t,\rho) = \bar{\gl}(\ga,\rho) \int_B \frac{1}{2a(\ga,\rho)} \Ind_{[\ga_0-a(\ga,\rho),\ga_0+a(\ga,\rho)]}(z)dz
\end{align*} 
for every $\ga  >0$, $B \in \cB(\R_{>0})$, $t \in [0,T]$ and $\rho \in L^1(\R)$. Here, we choose $ \bar{\gl}(\ga,\rho)  = \gl_0 +(\gl_1-\gl_0)V(\ga,\rho)$ for $\gl_0 = 3$ as the minimal and $\gl_1 = 10$ as the maximal rate, $a(\ga,\rho) =\sqrt{ \frac{9}{2 \cdot 10^3}(V(\ga,\rho)+1)}$ with $V(\ga,\rho) = \int_0^1 \Ind_{\rho(x) \geq \ga} dx$ and $\ga_0 = 0.4$. The functional $V(\ga,\rho)$ describes the portion of $[0,1]$, which is above the actual critical density $\ga$ and always lies in between zero and one. 
To study the free phase, we use an initial condition in the form of $\rho_0(x) = (0.05+0.4 \max\{\sin(x),0\})e^{-\frac{|x|}{100}}$. A sample of the density flux relation at $x = 0$ as well as at $x = 1$ is shown in figure \ref{fig:SampleDensFluxLWRFree} given at the times $t = 0,0.1,\dots,50$. We observe a low scattering as expected. Contrary, in figure \ref{fig:SampleDensFluxLWRCongested} a sample with initial condition $\rho_0(x) = (0.4+ \max\{\sin(x),0\})e^{-\frac{|x|}{100}}$, i.e.\ congested case, is shown resulting in high scattering. Finally, in figure \ref{fig:SampleDensAndFluxLWRCongested} the time evolution of the density and flux at $x =0$ in the congested case is shown. The density is not severely affected by the variation in $\ga$ compared to the flux.

\begin{figure}[htb!]
\begin{subfigure}[c]{0.5\textwidth}
\includegraphics[width=0.95\textwidth]{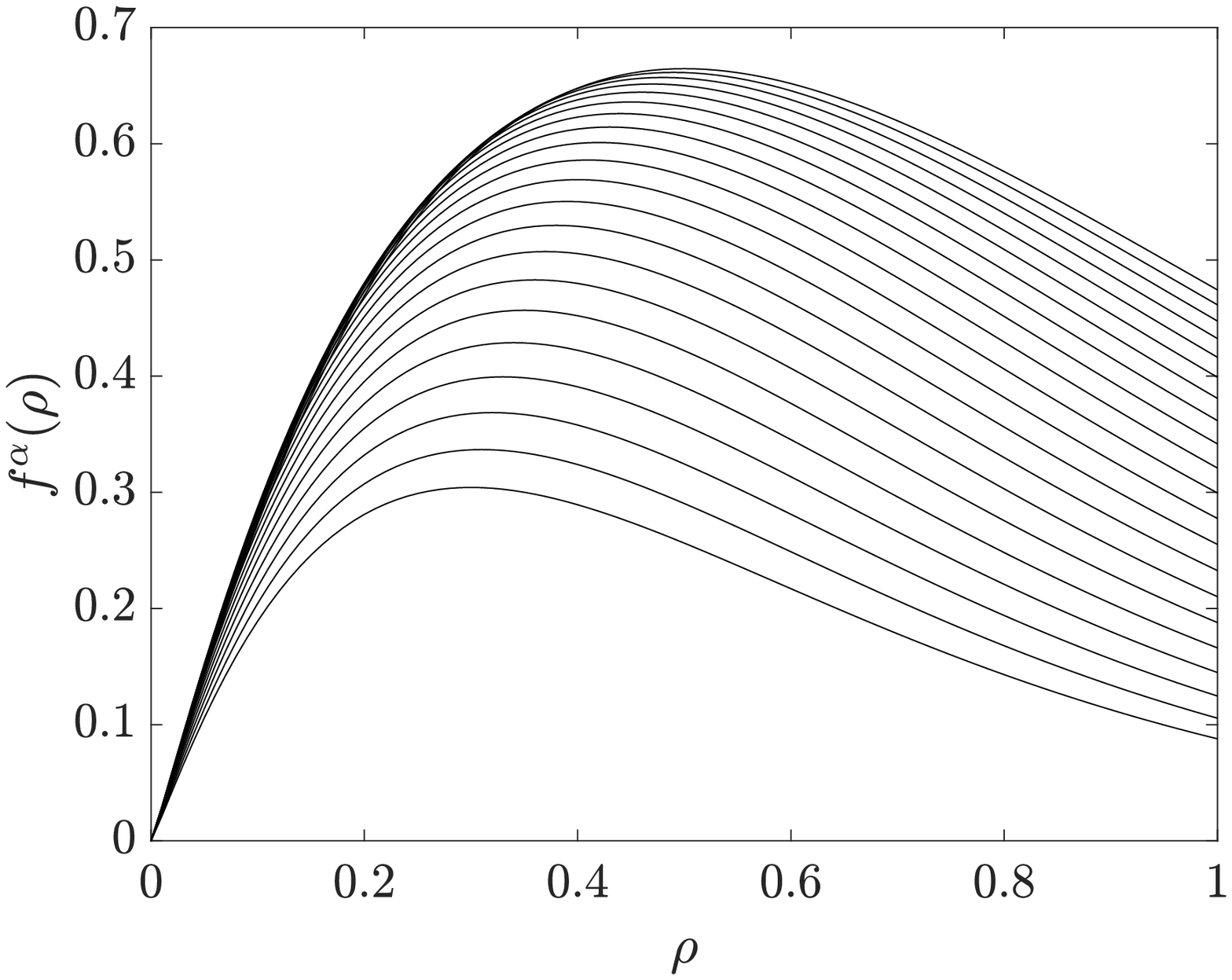}
\subcaption{Flux functions for different values\\of $\ga\in [0.3,0.5]$}
\label{fig:FluxFunctionsLWR1}
\end{subfigure}
\begin{subfigure}[c]{0.5\textwidth}
\includegraphics[width=0.95\textwidth]{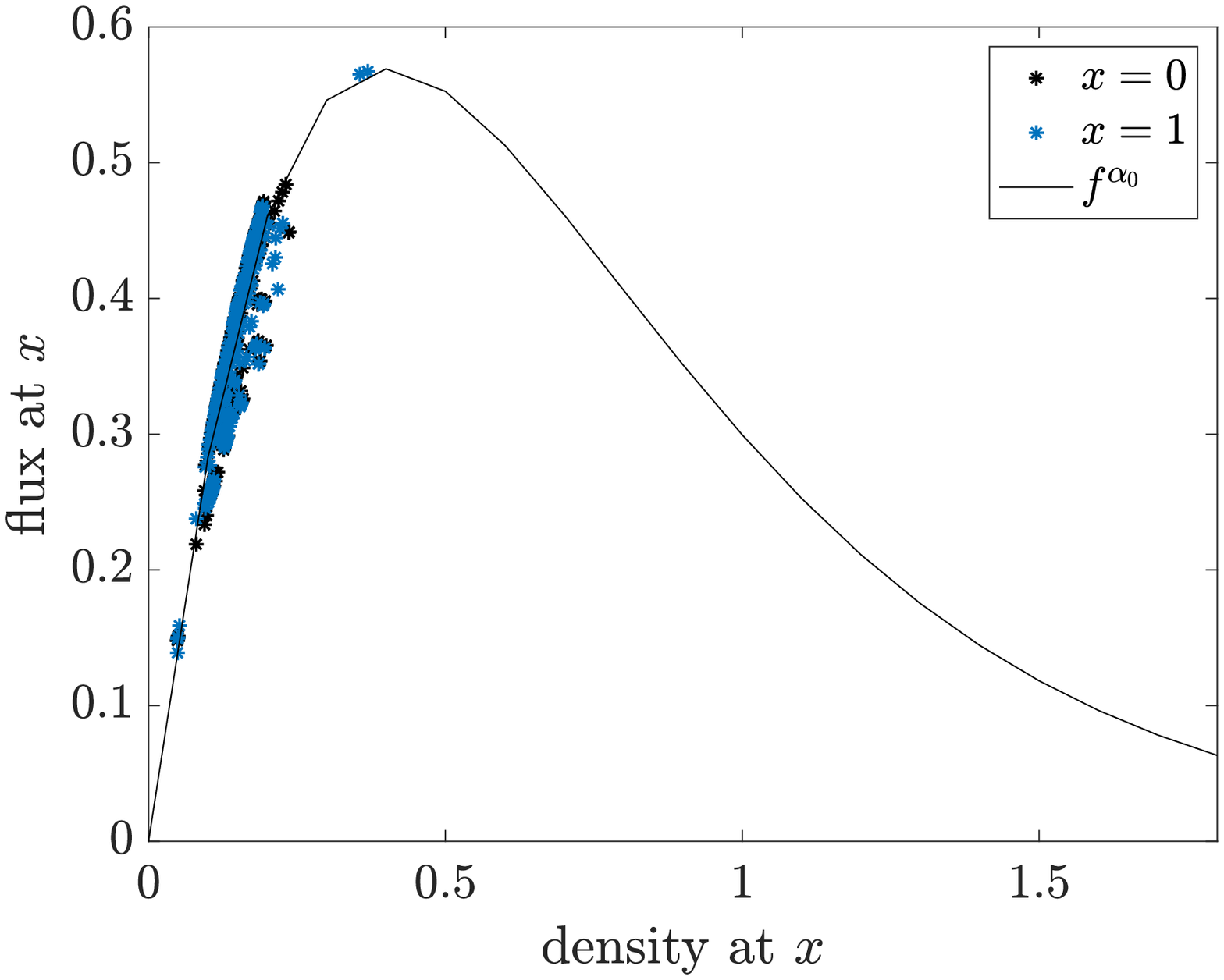}
\subcaption{Sample densities and fluxes at $x = 0$ \\and $x=1$ in free phase}
\label{fig:SampleDensFluxLWRFree}
\end{subfigure}\\
\begin{subfigure}[c]{0.5\textwidth}
\includegraphics[width=0.95\textwidth]{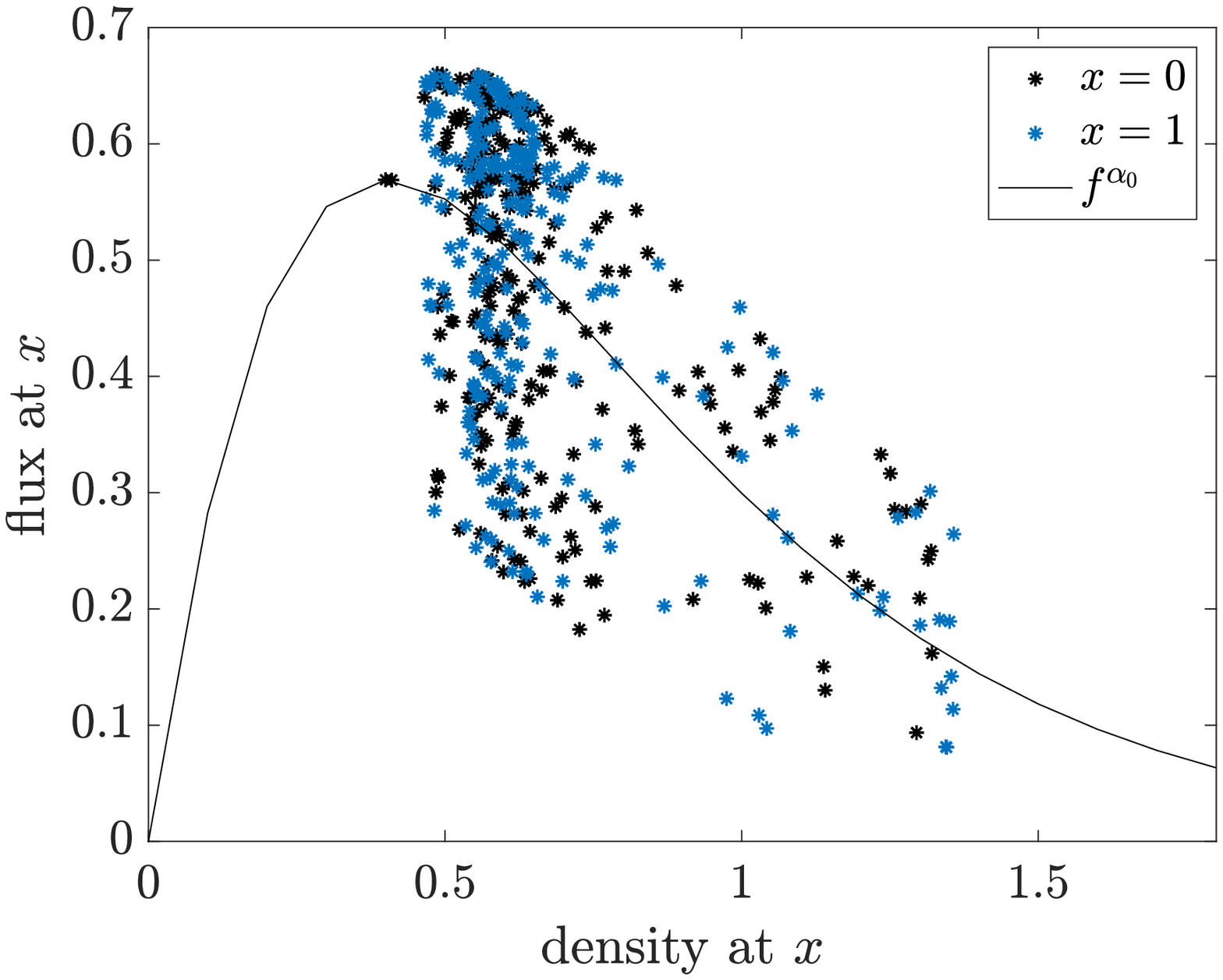}
\subcaption{Sample densities and fluxes at $x = 0$ \\and $x=1$ in congested phase}
\label{fig:SampleDensFluxLWRCongested}
\end{subfigure}
\begin{subfigure}[c]{0.5\textwidth}
\includegraphics[width=0.95\textwidth]{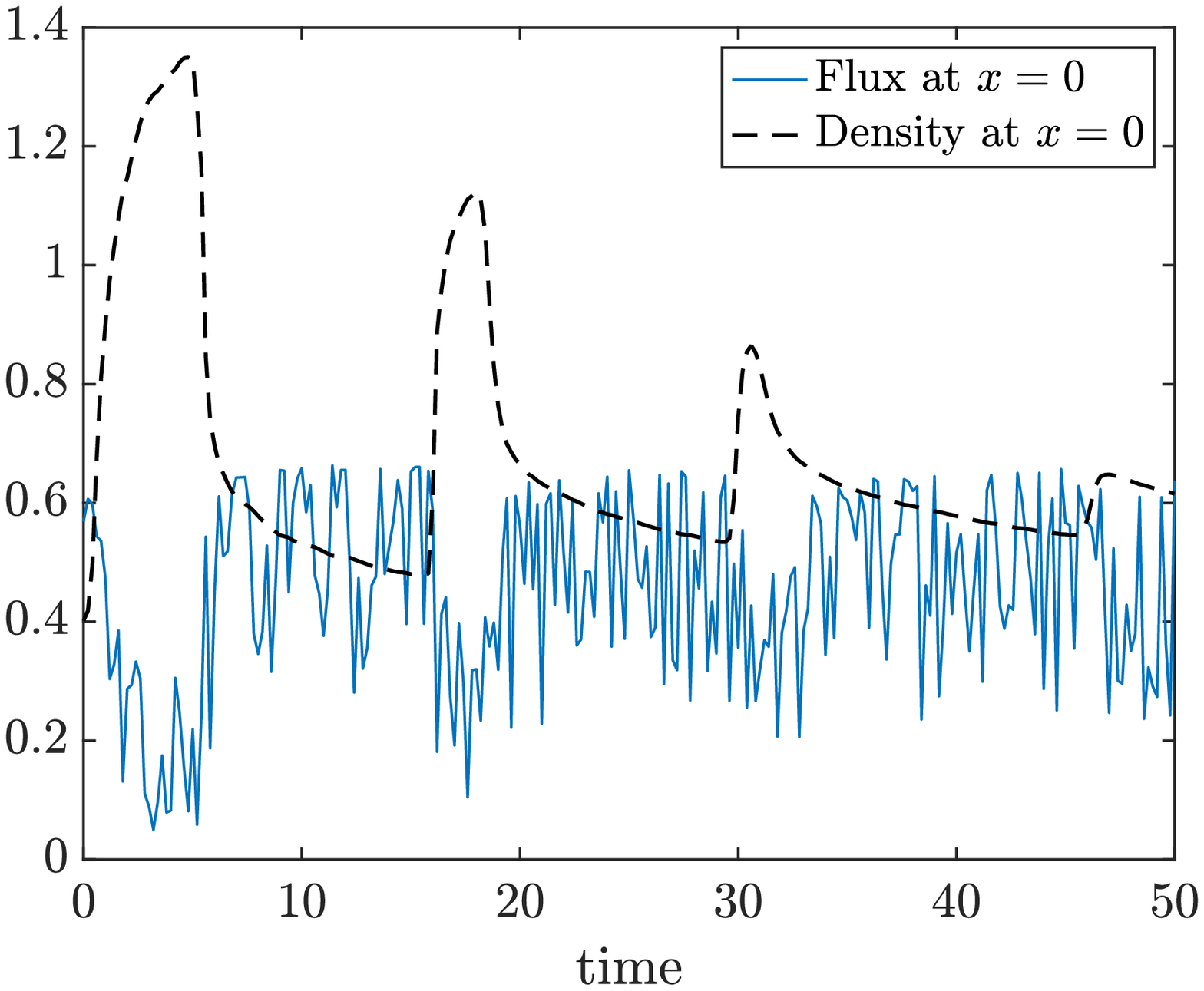}
\subcaption{Sample densities and fluxes at $x = 0$\\ in congested phase}
\label{fig:SampleDensAndFluxLWRCongested}
\end{subfigure}
\end{figure}

\section{Conclusions}
\label{sec:conclusions}
We have successfully incorporated random flux functions for scalar conservation laws in the sense of PDMPs. Additionally, we derived a sufficient condition for an arbitrary family of Lipschitz continuous flux functions such that we can guarantee the existence of a PDMP. The motivation of scattering effects in macroscopic models has been recovered in numerical simulation results in the case of a production and traffic flow model.

To cover more complex dynamics, like space dependent flux functions, the theory can be extended in a suitable way as future research. This can be relevant to model traffic accidents and models, where spatial events can happen. Additionally, systems of conservation laws should be examined as deterministic dynamics for PDMPs since the extension to $L^1$ solutions is not straightforward anymore.

\appendix

\bibliographystyle{siamplain}
\bibliography{references}
\end{document}